\documentclass[12pt, twoside, leqno]{article}
\usepackage{cite}
\usepackage{indentfirst}
\usepackage[T1]{fontenc}
\usepackage{enumitem}
\usepackage{amsthm,amsmath}
\usepackage{amssymb}
\newtheorem{thm}{Theorem}[section]
\newtheorem{lemma}{Lemma}[section]

\newtheorem*{xrem}{Remark}
\numberwithin{equation}{section}
\begin{document}
	\allowdisplaybreaks[3]
	\baselineskip=17pt
	\title{\bf Some new curious  congruences involving multiple harmonic sums}
	\author{Rong Ma\\
		School of Mathematics and Statistics,\\
		Northwestern
		Polytechnical University\\
		Xi'an, Shaanxi, 710072,\\
		People's Republic of China\\
		E-mail: marong@nwpu.edu.cn\\
		\and 
		Li Ni\\
		School of Mathematics and Statistics, \\
		Northwestern
		Polytechnical University\\
		Xi'an, Shaanxi, 710072,\\
		People's Republic of China\\
		E-mail: lini@mail.nwpu.edu.cn
		}
	\date{}
	\maketitle
	\renewcommand{\thefootnote}{}
	\renewcommand{\thefootnote}{\arabic{footnote}}
	\setcounter{footnote}{0}
	
	\begin{abstract}
	It is significant to study congruences involving multiple harmonic sums. Let $p$ be an odd prime, in recent years, the following curious congruence $$
	 \sum_{\substack{i+j+k=p \\ i, j, k>0}} \frac{1}{i j k} \equiv-2 B_{p-3}\pmod p
	$$	has been generalized along different directions, where $B_n$ denote the $n$th Bernoulli number. In this paper, we obtain several new generalizations of the above congruence by applying congruences involving  multiple harmonic sums. For example, we have $$
	\sum_{\substack{k_1+k_2+\cdots+k_n=p \\  k_i >0, 1 \le i \le n}} \dfrac{(-1)^{k_1}\left(\dfrac{k_1}{3}\right)}{k_1 \cdots k_n} \equiv  \dfrac{(n-1)!}{n}\dfrac{2^{n-1}+1}{3\cdot6^{n-1}}B_{p-n}\left(\dfrac{1}{3}\right)\pmod p,
	$$
	where $n$ is even, $B_n(x)$ denote the Bernoulli polynomials.
	\end{abstract}

	\section{Introduction}
	Let $\mathbb{N}$ be the set of positive integers. For any $n$, $N \in \mathbb{N}$ and \textbf{s} $=(s_1, \cdots ,s_n) \in \mathbb{N}^n$, we define the multiple harmonic sums (MHS) by$$
	H_N(\textbf{s}) :=\sum\limits_{0<k_1<\cdots<k_n\le N}\frac{1}{k_{1}^{s_1}\cdots k_{n}^{s_n}}.$$\par 
	Since the middle of 1980s, multiple harmonic sums have played an important role in the study of mathematics and theoretical physics. From 2008 onwards, JianQiang Zhao in \cite{Z1, Z2} and Hoffman in \cite{H} independently obtained many congruences modulo  prime and prime powers for MHS. Subsequently, Roberto Tauraso in \cite{T} began to consider congruence properties of alternating multiple harmonic sums which are defined as follows. Let $n$ be any positive integer and  $(r_1, \cdots ,r_n) \in (\mathbb{Z^*})^n$. For any $N \ge n$, we define the alternating multiple harmonic sums as $$
		H(r_1, r_2,\cdots ,r_n; N) :=\sum\limits_{0<k_1<\cdots<k_n\le N}\prod_{i=1}^{n}\frac{\text{sgn}^{k_i}(r_i)}{k_{i}^{\mid r_i \mid }}.$$\par
		Besides, many number theorists also investigated congruences of MHS with coefficients involving invariant sequences (see \cite{S2}). For example, in 2010, LiLu Zhao and ZhiWei Sun (see \cite[Theorem 1.1]{ZS}) showed that for any positive odd integer $n$, prime $p$ such that $p>n+1$
		\begin{equation}\tag{1.1}
		\sum\limits_{0<k_1<\cdots<k_n<p } \frac{(-1)^{k_1}\left(\dfrac{k_1}{3}\right)}{k_1 \cdots k_n}\equiv 0 \pmod p.
		\end{equation}\par
	    Sandro Mattarei and Tauraso (see \cite[Theorem 4.1]{MT}) proved that for any positive integer $n$ and prime $p$ with $p>\max (n+1,3)$ 
	    \begin{equation}\tag{1.2}
	    \sum_{0<k_1<\cdots<k_n<p}\left(\frac{p-k_n}{3}\right) \frac{(-1)^{k_n}}{k_1 \cdots k_n} \equiv \begin{cases}-\dfrac{2^{n+1}+2}{6^{n+1}} p B_{p-n-1}(1 / 3) \pmod {p^2}, \text { if } n \text { is odd }, \\ -\dfrac{2^{n+1}+4}{n 6^n} B_{p-n}(1 / 3) \pmod p, \quad \text { if } n \text { is even },\end{cases}
	    \end{equation}
	    where $B_n(x)$ denote the Bernoulli polynomials defined by $$ \dfrac{ze^{xz}}{e^z-1}=\sum\limits_{n=0}^{\infty}\dfrac{B_n(x)}{n!}z^n\left(\mid z \mid <2 \pi\right).$$
	     \par
	    As an application of multiple harmonic sums' congruent properties, Zhao (see \cite[Corollary 4.2]{Z1}) first proved the following curious congruence
	    \begin{equation}\tag{1.3}
	     \sum_{\substack{i+j+k=p \\ i, j, k>0}} \frac{1}{i j k} \equiv-2 B_{p-3}\pmod p
	    \end{equation}
	    for any prime $p \ge 3$, where $B_n$ denotes the $n$th Bernoulli number given by $B_0=1$ and $B_n=\sum\limits_{j=0}^{n}B_j\left(\begin{array}{c}
	    n\\
	    j 
	    \end{array}\right)\left(n \ge 2\right)$. Later on, this congruence has been generalized along different directions. For example, Zhou and  Cai (see \cite{ZC}) generalized (1.3) by increasing the number of indices. Wang, Cai (see \cite{WC}) and Cai, Shen, Jia (see \cite{CSJ}) generalized the bound from $p$ to $p$-powers, and a product of two odd prime powers, respectively. In addtion, some scholars established and generalized congruences for alternating version of (1.3) (see \cite{SC}, \cite{CHQ}).\par
	    In view of congruences (1.1) and (1.2). In this paper, we shall prove the following theorems.
	    \begin{thm}
	    Let $n$ be a positive integer and $p$ a prime with $p >\max (n+1,3)$, then
	    \begin{multline}\notag
	    	\begin{aligned}
	    	&\sum\limits_{0<k_1<\cdots<k_n\le p-1}\dfrac{(-1)^{(p-k_n)}\left(\dfrac{p-k_n}{3}\right)}{k_{1}\cdots k_n^{2}}\\
	    	&\equiv \begin{cases}  -\dfrac{2^{n}+1}{3(n+1)6^n}B_{p-n-1}\left(\dfrac{1}{3}\right)\pmod p & \text {if $2 \nmid n$} , \\
	    		\sum\limits_{r=0}^{p-1-n}\left(\begin{array}{c}
	    			{p-n} \\
	    			r
	    		\end{array}\right) \dfrac{(2^{n+r+1}+4)B_rB_{p-n-r}(1 / 3)}{n(n+r)6^{n+r}}-\left(\dfrac{p}{3}\right)B_{p-n-1}\pmod p &  \text {if $2 \mid n$} .\end{cases}
	    	\end{aligned}
	    \end{multline}
	    \end{thm}
     \begin{thm}
    	Let $n \ge 2$ be a positive integer and $p$ a prime with $p > n+1$, then
    	\begin{multline}\notag
    		\begin{aligned}
    		&\sum_{\substack{k_1+k_2+\cdots+k_n=p \\  k_i >0, 1 \le i \le n}} \dfrac{(-1)^{k_1}\left(\dfrac{k_1}{3}\right)}{k_1 \cdots k_n}\\
    		&\equiv  \begin{cases}  \dfrac{(n-1)!}{n}\dfrac{2^{n-1}+1}{3\cdot6^{n-1}}B_{p-n}\left(\dfrac{1}{3}\right)\pmod p & \text {if $2 \mid n$} , \\
    			-(n-2)!\sum\limits_{r=0}^{p-n}\left(\begin{array}{c}
    				{p-n+1} \\
    				r
    			\end{array}\right) \dfrac{(2^{n+r}+4)B_rB_{p-n-r+1}(1 / 3)}{(n+r-1)6^{n+r-1}}\\
    		\quad {}+(n-1)!\left(\dfrac{p}{3}\right)B_{p-1-n}\pmod p &  \text {if $2 \nmid n$} .\end{cases}
    		\end{aligned}
    	\end{multline}
    \end{thm}
 \begin{thm}
	Let $n \ge 2$ be a positive integer and $p$ a prime with $p > {n+1}$, then
	\begin{multline}\notag
		\begin{aligned}
			&\sum_{\substack{k_1+k_2+\cdots+k_n=p \\  k_i >0, 1 \le i \le n}} \dfrac{(-1)^{k_1}\left(\left(\dfrac{k_1+1}{3}\right)-\left(\dfrac{k_1-1}{3}\right)\right)}{k_1 \cdots k_n}\\
			& \equiv  \begin{cases} 
				-(n-2)!\sum\limits_{r=0}^{p-n}\left(\begin{array}{c}
					{p-n+1} \\
					r
				\end{array}\right) \dfrac{(3^{n+r-2}-1)(2^{n+r-2}-1)B_rB_{p-n-r+1}}{(n+r-1)6^{n+r-2}} & \text {if $2 \mid n$} , \\
				-(n-1)!\left(1-\dfrac{(3^{n-1}-1)(2^{n-1}-1)}{2n 6^{n-1}}\right) B_{p-n}  & \text {if $2 \nmid n$}.\end{cases}\pmod p
		\end{aligned}
	\end{multline}
\end{thm}
	 \section{Basic lemmas}
	 In this section, we introduce the following lemmas that will be used later to prove the  theorems.   
	  \begin{lemma}[See \cite{S3}]
	 	Suppose that $k$, $p \in \mathbb {N}$ with $p>1$. If $x$, $y \in \mathbb{Z}_p$, then $p B_k(x) \in \mathbb{Z}_p$ and $\left(B_k(x)-\right.$ $\left.B_k(y)\right) / k \in \mathbb{Z}_p$. If $p$ is an odd prime such that $p-1 \nmid k$, then $B_k(x) / k \in \mathbb{Z}_p$.
	 \end{lemma}
     \begin{lemma}[See \cite{S3}]
     Let $p, m \in \mathbb{N}$ and $k, r \in \mathbb{Z}$ with $k \geqslant 0$. Then
     $$
     \sum_{\substack{x=0 \\ x \equiv r(\bmod m)}}^{p-1} x^k=\frac{m^k}{k+1}\left(B_{k+1}\left(\frac{p}{m}+\left\{\frac{r-p}{m}\right\}\right)-B_{k+1}\left(\left\{\frac{r}{m}\right\}\right)\right).
     $$
     \end{lemma}
 \begin{lemma}[See \cite{MOS}]
 	Let $n \in \mathbb{N}$, then
 	$$
 	B_{2 n}\left(\frac{1}{2}\right)=\left(2^{1-n}-1\right)B_{2n}, \quad 
 	B_{2 n}\left(\frac{1}{4}\right)=B_{2 n}\left(\frac{3}{4}\right)=\frac{2-2^{2 n}}{4^{2 n}} B_{2 n},$$ 
 	and
 	$$B_{2 n}\left(\frac{1}{3}\right)=B_{2 n}\left(\frac{2}{3}\right)=\frac{3-3^{2 n}}{2 \cdot 3^{2 n}} B_{2 n},
 	$$
    $$
 	\quad B_{2 n}\left(\frac{1}{6}\right)=B_{2 n}\left(\frac{5}{6}\right)=\frac{\left(2-2^{2 n}\right)\left(3-3^{2 n}\right)}{2 \cdot 6^{2 n}} B_{2 n} .
 	$$
 	Moreover, through the formulas $B_n(1-x)=(-1)^n B_n(x)$ and $ B_n(ax)=a^{m-1}\sum\limits_{k=0}^{a-1}B_n\left(x+\dfrac{k}{a}\right)$ for $m$, $a>0$, we have
 	$$\begin{aligned}
 		& B_{2n+1}\left(\frac{1}{2}\right)=0, \quad B_{2n+1}\left(\frac{2}{3}\right)=-B_{2n+1}\left(\frac{1}{3}\right), \\
 		& B_{2n+1}\left(\frac{1}{6}\right)=-B_{2n+1}\left(\frac{5}{6}\right)=\left(1+2^{1-(2n+1)}\right) B_{2n+1}\left(\frac{1}{3}\right) .
 	\end{aligned}$$
 \end{lemma}
 \begin{lemma}[See \cite{MT}]
 	Let $n$ be a positive integer and let $p$ be a prime with $p>n+1$. Then we have the polynomial congruence
 	$$
 	\sum\limits_{0<k_1<\cdots<k_n\le p-1}\frac{x^{k_n}}{k_{1}\cdots k_n} \equiv(-1)^{n-1} \sum_{k=1}^{p-1} \frac{(1-x)^k}{k^n} \left(\bmod p\right).
 	$$
 \end{lemma}
 \begin{lemma}[See \cite{TZ}]
  Let $a$, $b \in \mathbb{N}$ and a prime $p \ge a+b+2$. If $a+b$ is odd then we have $$
  \sum\limits_{0<k_1 \le k_2\le p-1}\frac{1}{k_{1}^{a}k_{2}^{b}} \equiv H(a, b; p-1) \equiv \frac{(-1)^b}{a+b}\left(\begin{array}{c}
  	a+b\\
  	a 
  \end{array}\right)B_{p-a-b}\pmod p.$$
 If $a+b$ is even then we have$$
 \sum\limits_{0<k_1 \le k_2\le p-1}\frac{1}{k_{1}^{a}k_{2}^{b}} \equiv H(a, b; p-1) \equiv 0 \pmod p.$$
 \end{lemma} 
\begin{lemma}
	Let $n$ be a positive integer and $p>n+1$ a prime, then we get
	\begin{multline}\notag
		\sum\limits_{0<k_1<\cdots<k_n\le p-1}\frac{x^{k_n}}{k_{1}\cdots k_n^{2}} \equiv \begin{cases}  \sum\limits_{0<i_1 \le i_2\le p-1}\dfrac{(1-x)^{i_1}}{i_1i_2^n}\pmod p & \text {if $2 \nmid n$} , \\
		- \sum\limits_{0<i_1 \le i_2\le p-1}\dfrac{(1-x)^{i_1}}{i_1i_2^n}+B_{p-1-n}\pmod p &  \text {if $2 \mid n$} .\end{cases}
	\end{multline}
\end{lemma}
\begin{proof}
	Taking the derivative of the left-hand side  above  and applying Lemma 2.4, we have 
	\begin{align*}
	&x\frac{d}{dx}\left(	\sum\limits_{0<k_1<\cdots<k_n\le p-1}\frac{x^{k_n}}{k_{1}\cdots k_n^{2}}\right)\\
	\quad {}&=	\sum\limits_{0<k_1<\cdots<k_n\le p-1}\frac{x^{k_n}}{k_{1}\cdots k_n}\equiv (-1)^{n-1} \sum_{k=1}^{p-1} \frac{(1-x)^k}{k^n} \pmod p.
	\end{align*} 
  Observe that 
 \begin{align*}
 (-1)^{n-1} \sum_{k=1}^{p-1} \frac{(1-x)^k}{k^n}& =(-1)^{n-1} \sum_{k=1}^{p-1} \frac{1}{k^n}\sum\limits_{r=0}^{k}(-x)^k\left(\begin{array}{c}
 	k\\
 r
 \end{array}\right)\\
&=(-1)^{n-1} \sum_{k=1}^{p-1} \frac{1}{k^n}+(-1)^{n-1} \sum_{k=1}^{p-1} \frac{1}{k^n}\sum\limits_{r=1}^{k}(-x)^r\left(\begin{array}{c}
	k\\
	r
\end{array}\right)\\
& \equiv (-1)^{n-1} \sum_{k=1}^{p-1} \frac{1}{k^n}\sum\limits_{r=1}^{k}(-x)^r\left(\begin{array}{c}
	k\\
	r
\end{array}\right)\pmod p,
 \end{align*}
where we use the congruence $ \sum\limits_{k=1}^{p-1} \dfrac{1}{k^n} \equiv 0 \pmod p$ (see \cite{S1}). Hence we can easily get$$
x\frac{d}{dx}\left(	\sum\limits_{0<k_1<\cdots<k_n\le p-1}\frac{x^{k_n}}{k_{1}\cdots k_n^{2}}-(-1)^{n-1} \sum_{k=1}^{p-1} \frac{1}{k^n}\sum\limits_{r=1}^{k}\frac{(-x)^r}{r}\left(\begin{array}{c}
	k\\
	r
\end{array}\right)\right) \equiv 0 \pmod p.$$
Thus $\sum\limits_{0<k_1<\cdots<k_n\le p-1}\dfrac{x^{k_n}}{k_{1}\cdots k_n^{2}}-(-1)^{n-1} \sum\limits_{k=1}^{p-1} \dfrac{1}{k^n}\sum\limits_{r=1}^{k}\dfrac{(-x)^r}{r}\left(\begin{array}{c}
	k\\
	r
\end{array}\right)\equiv c \pmod p$ for some constant $c$. Replacing $x$ with $0$ we obtain
\begin{equation}
\sum\limits_{0<k_1<\cdots<k_n\le p-1}\frac{x^{k_n}}{k_{1}\cdots k_n^{2}}\equiv (-1)^{n-1} \sum_{k=1}^{p-1} \frac{1}{k^n}\sum\limits_{r=1}^{k}\frac{(-x)^r}{r}\left(\begin{array}{c}
	k\\
	r
\end{array}\right) \pmod p.
\end{equation}
Since $\int_{0}^{1}u^{r-1}du=\dfrac{1}{r}$, taking $v=1-xu$ we see that$$
\begin{aligned}
	\sum_{r=1}^{k}  \frac{(-x)^{r}}{r}\left(\begin{array}{l}
		k \\
		r
	\end{array}\right)& =\int_0^1 \sum_{r=1}^{k}(-x)^{r}\left(\begin{array}{l}
		k \\
		r
	\end{array}\right) u^{r-1} d u=\int_0^1 \frac{(1-x u)^k-1}{u} d u \\
	& =\int_1^{1-x}\sum_{i=1}^k v^{i-1} d v 
	 =\sum_{i=1}^{k} \frac{(1-x)^i-1}{i}.
\end{aligned}$$
Combining the identity above with (2.1) and Lemma 2.5, we complete the proof of Lemma 2.6.
\end{proof}	
\begin{xrem}
\rm{If we set $n=1, x=2 $ in Lemma 2.6, we can easily get $$
	\sum\limits_{0<i \le j\le p-1}\dfrac{(-1)^{i}}{ij}\equiv H(-1,1; p-1) \equiv -{{q_p}^2(2)} \pmod p
	$$ by using the congruence  $\sum\limits_{k=1}^{p-1} \dfrac{2^k}{k^2} \equiv-{{q_p}^2(2)}
	\pmod p$ (see \cite{G}), where ${q_p(2)}$ is Euler's quotient of 2 with base $p$.  }
\end{xrem}
\begin{lemma}
Let $n$ be a positive integer and $p$ be a prime with $p > \max(n+1, 3)$, then we have
\begin{equation*}
\sum\limits_{i=1}^{p-1}\frac{(-1)^{(p+i)}\left(\frac{p+i}{3}\right)}{i^n} \equiv \begin{cases}-\dfrac{2^{n+1}+4}{n 6^n} B_{p-n}(1 / 3)\pmod p & \text {if $2 \mid n$} , \\
0 \pmod p
 &  \text {if $2 \nmid n$} .\end{cases}
\end{equation*}
\end{lemma}
\begin{proof}
	For prime  $p > \max(n+1, 3),$ according to  Euler's theorem and Lemma 2.2  we have $$
	\begin{aligned}
	 \sum_{\substack{i=1 \\ i \equiv r(\bmod 6)}}^{p-1} \frac{1}{i^n} \equiv& \sum_{\substack{i=1 \\ i \equiv r(\bmod 6)}}^{p-1}i^{p-n-1}\\
	 \equiv& -\frac{1}{n6^n}\left(B_{p-n}\left(\frac{p}{6}+\left\{\frac{r-p}{6}\right\}\right)-B_{p-n}\left(\left\{\frac{r}{6}\right\}\right)\right)\pmod p,	
	\end{aligned},
	$$where $\{x\}$ is the decimal part of $x$. By using the following  formula for Bernoulli polynomials
	$$
	B_n(x+y)=\sum_{r=0}^n\left(\begin{array}{l}n \\ r\end{array}\right) B_{n-r}(y) x^r$$
	 and Lemma 2.1, we have 
	 \begin{equation}
	 \sum_{\substack{i=1 \\ i \equiv r(\bmod 6)}}^{p-1} \frac{1}{i^n} \equiv \frac{1}{n6^n}\left(B_{p-n}\left(\left\{\frac{r}{6}\right\}\right)-B_{p-n}\left(\left\{\frac{r-p}{6}\right\}\right)\right)\pmod p.
	 \end{equation}\par
	  On the other hand, $(-1)^{(p+i)}(\dfrac{p+i}{3})$ takes the values $0, -1, -1, 0, 1, 1$ according as $p+i \equiv 0, 1, 2, 3, 4, 5\pmod 6.$ When $n$ is even, $p \equiv 1 \pmod 6$, we have
	  $$\sum\limits_{i=1}^{p-1}\frac{(-1)^{(p+i)}\left(\frac{p+i}{3}\right)}{i^n}= -\sum_{\substack{i=1 \\ i \equiv 0\pmod 6}}^{p-1} \frac{1}{i^n}-\sum_{\substack{i=1 \\ i \equiv 1\pmod 6}}^{p-1} \frac{1}{i^n}+ \sum_{\substack{i=1 \\ i \equiv 3\pmod 6}}^{p-1} \frac{1}{i^n}+ \sum_{\substack{i=1 \\ i \equiv 4\pmod 6}}^{p-1} \frac{1}{i^n}.$$
	  Combining the above identity, (2.2) and Lemma 2.3, we obtain$$
	  \sum\limits_{i=1}^{p-1}\frac{(-1)^{(p+i)}\left(\frac{p+i}{3}\right)}{i^n} \equiv-\dfrac{2^{n+1}+4}{n 6^n} B_{p-n}(1 / 3) \pmod p.$$
	  Similarly, when $p \equiv -1 \pmod 6$, we have
	    $$\sum\limits_{i=1}^{p-1}\frac{(-1)^{(p+i)}\left(\frac{p+i}{3}\right)}{i^n}= -\sum_{\substack{i=1 \\ i \equiv 2\pmod 6}}^{p-1} \frac{1}{i^n}-\sum_{\substack{i=1 \\ i \equiv 3\pmod 6}}^{p-1} \frac{1}{i^n}+ \sum_{\substack{i=1 \\ i \equiv 5\pmod 6}}^{p-1} \frac{1}{i^n}+ \sum_{\substack{i=1 \\ i \equiv 0\pmod 6}}^{p-1} \frac{1}{i^n}.$$ Applying (2.2) and Lemma 2.3 again, we obtain the same congruence as the case of $p \equiv 1 \pmod 6$. Thus we prove the case where $n$ is even.\par
	  When $n$ is odd, $p-n$ is even, we can  prove the second congruence of Lemma 2.7 in the similar way .
\end{proof}
	\section{Proofs of the theorems}
	\textbf{Proof of Theorem 1.1}\quad Let $\omega$ be a complex primitive cubic root of unity. For any  integer $i$, we note that  $\left(\dfrac{i}{3}\right)=(\omega^i-\omega^{-i})/\left(\omega-\omega^{-1}\right)$ and $1-(-\omega)^{-1}=-\omega$. Hence, applying Lemma 2.6 and Lemma 2.5 we have 
	\begin{equation}
	\begin{aligned}
		&\sum\limits_{0<k_1<\cdots<k_n\le p-1}\frac{(-1)^{(p-k_n)}\left(\frac{p-k_n}{3}\right)}{k_{1}\cdots k_n^{2}}\\
		&=\frac{1}{\omega-\omega^{-1}}\left((-\omega)^p \sum\limits_{0<k_1<\cdots<k_n\le p-1}\frac{(-\omega)^{-k_n}}{k_{1}\cdots k_n^{2}} -(-\omega)^{-p} \sum\limits_{0<k_1<\cdots<k_n\le p-1}\frac{(-\omega)^{k_n}}{k_{1}\cdots k_n^{2}}\right)\\
		&\equiv \frac{(-1)^{n-1}}{\omega-\omega^{-1}}\left((-\omega)^p  \sum\limits_{0<i_1 \le i_2\le p-1}\frac{(-\omega)^{i_1}-1}{i_1i_2^n}-(-\omega)^{-p} \sum\limits_{0<i_1 \le i_2\le p-1}\frac{(-\omega)^{-i_1}-1}{i_1i_2^n}\right)\\
		&\equiv (-1)^{n-1}\left(\sum\limits_{0<i_1 \le i_2\le p-1}\frac{(-1)^{(p+i_1)}\left(\frac{p+i_1}{3}\right)}{i_1i_2^n}+\left(\frac{p}{3}\right)\sum\limits_{0<i_1 \le i_2\le p-1}\frac{1}{i_1i_2^n}\right)\\
		&\equiv \begin{cases}  \sum\limits_{0<i_1 \le i_2\le p-1}\dfrac{(-1)^{(p+i_1)}\left(\dfrac{p+i_1}{3}\right)}{i_1i_2^n}\pmod p & \text {if $2 \nmid n$} , \\
			- \sum\limits_{0<i_1 \le i_2\le p-1}\dfrac{(-1)^{(p+i_1)}\left(\dfrac{p+i_1}{3}\right)}{i_1i_2^n}-\left(\dfrac{p}{3}\right)B_{p-1-n}\pmod p &  \text {if $2 \mid n$} .\end{cases}
	\end{aligned}
	\end{equation}
	When $n$ is odd. Under substitutions $i_1 \rightarrow p-i_2$ and $i_2 \rightarrow p-i_1$, we obtain
	\begin{align*}
		\sum\limits_{0<i_1 \le i_2\le p-1}\dfrac{(-1)^{(p+i_1)}\left(\dfrac{p+i_1}{3}\right)}{i_1i_2^n}&\equiv 	\sum\limits_{0<i_1 \le i_2\le p-1}\dfrac{(-1)^{(2p-i_2)}\left(\dfrac{2p-i_2}{3}\right)}{i_1^ni_2}\\
		&\equiv \sum\limits_{0<i_1 \le i_2\le p-1}\dfrac{(-1)^{(p+i_2)}\left(\dfrac{p+i_2}{3}\right)}{i_1^ni_2}\pmod p.
	\end{align*}
Thus, by  $ \sum\limits_{k=1}^{p-1} \dfrac{1}{k^n} \equiv 0 \pmod p (n<p-1)$ and Lemma 2.7, we have 
 \begin{equation}\notag
	\begin{aligned}
		& \sum\limits_{0<i_1 \le i_2\le p-1}\dfrac{(-1)^{(p+i_1)}\left(\dfrac{p+i_1}{3}\right)}{i_1i_2^n}\\ \qquad {}&\equiv \frac{1}{2}\left(\sum\limits_{0<i_1, i_2\le p-1}\dfrac{(-1)^{(p+i_1)}\left(\dfrac{p+i_1}{3}\right)}{i_1i_2^n}+\sum\limits_{i=1}^{p-1}\dfrac{(-1)^{(p+i)}\left(\dfrac{p+i}{3}\right)}{i^{n+1}}\right)\\
		& \equiv \frac{1}{2}\sum\limits_{i=1}^{p-1}\dfrac{(-1)^{(p+i)}\left(\dfrac{p+i}{3}\right)}{i^{n+1}}\\
		&\equiv -\frac{2^{n}+1}{3(n+1)6^n}B_{p-n-1}\left(\frac{1}{3}\right)\pmod p.
	\end{aligned}
\end{equation}
Now using (3.1), we complete the proof of the first case of theorem 1.1.\par
When $n$ is even. Note that
\begin{equation}
	\begin{aligned}
		&\sum\limits_{0<i_1 \le i_2\le p-1}\dfrac{(-1)^{(p+i_1)}\left(\dfrac{p+i_1}{3}\right)}{i_1i_2^n}\\
		&=	\sum\limits_{0<i_1,  i_2\le p-1}\dfrac{(-1)^{(p+i_1)}\left(\dfrac{p+i_1}{3}\right)}{i_1i_2^n}-\sum\limits_{0<i_2 <i_1\le p-1}\dfrac{(-1)^{(p+i_1)}\left(\dfrac{p+i_1}{3}\right)}{i_1i_2^n}\\
		&\equiv -\sum\limits_{0<j <i\le p-1}\dfrac{(-1)^{(p+i)}\left(\dfrac{p+i}{3}\right)}{ij^n}\\
		&\equiv -\sum\limits_{0<i\le p-1}\dfrac{(-1)^{(p+i)}\left(\dfrac{p+i}{3}\right)}{i}\sum\limits_{0<j\le i-1}\dfrac{1}{j^n}
		\pmod p.
	\end{aligned}
\end{equation}
From the difference equation of Bernoulli polynomials $B_n(x+1)-B_n(x)=nx^{n-1}\left(n \ge 1\right)$, we can obtain the following well-known identity$$
\sum_{x=1}^{n-1} x^k=\sum_{r=0}^k\left(\begin{array}{c}
	k+1 \\
	r
\end{array}\right) \frac{B_r}{k+1} n^{k+1-r}, \quad \forall n, k \geq 1.$$
Applying  Euler's theorem  and the above identity to the last congruence of (3.2), we get $$
\begin{aligned}
	&\sum\limits_{0<i_1 \le i_2\le p-1}\dfrac{(-1)^{(p+i_1)}\left(\dfrac{p+i_1}{3}\right)}{i_1i_2^n}\\
	&\equiv -\sum_{r=0}^{p-1-n}\left(\begin{array}{c}
		{p-n} \\
		r
	\end{array}\right) \frac{B_r}{{p-n}}\sum\limits_{0<i\le p-1}\dfrac{(-1)^{(p+i)}\left(\dfrac{p+i}{3}\right)}{i^{n+r}} \pmod p.
\end{aligned}$$
Then we can prove the second congruence of  Theorem 1.1 by using Lemma 2.7 and (3.1), hence Theorem 1.1 is proved. \qedsymbol\\
\\
\textbf{Proof of Theorem 1.2}\quad Let $l=k_2+\cdots +k_n$. It's easy to see that
\begin{multline}
	\begin{aligned}
		\sum_{\substack{k_1+k_2+\cdots+k_n=p \\  k_i >0, 1 \le i \le n}} \dfrac{(-1)^{k_1}\left(\dfrac{k_1}{3}\right)}{k_1 \cdots k_n}&=\sum_{\substack{k_2+\cdots+k_n=l<p \\ 0< k_i, 2 \le i \le n}} \dfrac{(-1)^{(p-l)}\left(\dfrac{p-l}{3}\right)}{(p-l)k_2 \cdots k_n}\\
		&\equiv -\sum_{\substack{k_2+\cdots+k_n=l<p \\ 0< k_i, 2 \le i \le n}} \dfrac{(-1)^{p-l}\left(\dfrac{p-l}{3}\right)}{lk_2 \cdots k_n}\pmod p\\
		&= -(n-1)\sum_{\substack{k_2+\cdots+k_{n-1}=l_1<l<p \\ 0< k_i, 2 \le i \le {n-1}}} \dfrac{(-1)^{p-l}\left(\dfrac{p-l}{3}\right)}{l^2k_2 \cdots k_{n-1}}\\
		&=-(n-1)(n-2)\sum_{\substack{k_2+\cdots+k_{n-2}=l_2<l_1<l<p \\ 0< k_i, 2 \le i \le _n-2}} \dfrac{(-1)^{p-l}\left(\dfrac{p-l}{3}\right)}{l^2l_1k_2 \cdots k_{n-2}}\\
		&=\cdots= -(n-1)!\sum_{\substack 0<l_{n-2}<\cdots <l_1<l<p } \dfrac{(-1)^{p-l}\left(\dfrac{p-l}{3}\right)}{l^2l_1 \cdots l_{n-2}}.
	\end{aligned}
\end{multline}
Then  we can immediately obtain Theorem 1.2 by applying Theorem 1.1 .\qedsymbol\\
\\
\textbf{Proof of Theorem 1.3}\quad Similar to Theorem 1.2, for positive integer $n \ge 2$ and prime $p>n+1$,  we have 
\begin{multline}
\begin{aligned}
	&\sum_{\substack{k_1+k_2+\cdots+k_n=p \\  k_i >0, 1 \le i \le n}} \dfrac{(-1)^{k_1}\left(\left(\dfrac{k_1+1}{3}\right)-\left(\dfrac{k_1-1}{3}\right)\right)}{k_1 \cdots k_n}\\ &\equiv -(n-1)!\sum_{\substack 0<l_{n-2}<\cdots <l_1<l<p } \dfrac{(-1)^{p-l}\left(\left(\dfrac{p-l+1}{3}\right)-\left(\dfrac{p-l-1}{3}\right)\right)}{l^2l_1 \cdots l_{n-2}} \pmod p.
\end{aligned}
\end{multline}\par
Let $\omega$ be a complex primitive cubic root of unity. we observe that $\left(\dfrac{i+1}{3}\right)-\left(\dfrac{i-1}{3}\right)={\omega}^i+{\omega}^{-i}$ for any integer $i$. Using the same method as (3.1) and the fact that $\left(\dfrac{p+1}{3}\right)-\left(\dfrac{p-1}{3}\right)=-1$ for prime $p>3$, we obtain
\begin{multline}
\begin{aligned}
	&\sum_{\substack 0<l_{n-2}<\cdots <l_1<l<p } \dfrac{(-1)^{p-l}\left(\left(\dfrac{p-l+1}{3}\right)-\left(\dfrac{p-l-1}{3}\right)\right)}{l^2l_1 \cdots l_{n-2}} \\
	& \equiv  \begin{cases}  \sum\limits_{0<i_1 \le i_2\le p-1}\dfrac{(-1)^{(p+i_1)}\left(\left(\dfrac{p+i_1+1}{3}\right)-\left(\dfrac{p+i_1-1}{3}\right)\right)}{i_1i_2^{n-1}}\pmod p & \text {if $2 \mid n$} , \\
		- \sum\limits_{0<i_1 \le i_2\le p-1}\dfrac{(-1)^{(p+i_1)}\left(\left(\dfrac{p+i_1+1}{3}\right)-\left(\dfrac{p+i_1-1}{3}\right)\right)}{i_1i_2^{n-1}}\\
		\quad {}+B_{p-n}\pmod p &  \text {if $2 \nmid n$} .\end{cases}
\end{aligned}
\end{multline}
\par 
 On the other hand, $(-1)^{(p+i)}\left(\left(\dfrac{p+i+1}{3}\right)-\left(\dfrac{p+i-1}{3}\right)\right)$ takes the values $2$, $1$, $-1$, $-2$, $-1$, $1$  as $p+i \equiv 0, 1, 2, 3, 4, 5\pmod 6.$ Referring to the proof of Lemma 2.7, we can get 
 \begin{multline}
 \begin{aligned}
 	&\sum\limits_{i=1}^{p-1}\frac{(-1)^{(p+i)}\left(\left(\dfrac{p+i+1}{3}\right)-\left(\dfrac{p+i-1}{3}\right)\right)}{i^n}\\
 	\qquad {}&\equiv \begin{cases}\dfrac{(3^{n-1}-1)(2^{n-1}-1)}{n 6^{n-1}} B_{p-n}\pmod p & \text {if $2 \nmid n$} , \\
 		0 \pmod p
 		&  \text {if $2 \mid n$} .\end{cases}
 \end{aligned}
 \end{multline}
 Applying (3.6) and  (3.5) , we can obtain the following congruences in a way similar to Theorem 1.1,
 \begin{multline}
	\begin{aligned}
		&\sum_{\substack 0<l_{n-2}<\cdots <l_1<l<p } \dfrac{(-1)^{p-l}\left(\left(\dfrac{p-l+1}{3}\right)-\left(\dfrac{p-l-1}{3}\right)\right)}{l^2l_1 \cdots l_{n-2}} \\
		& \equiv  \begin{cases} 
		\sum\limits_{r=0}^{p-n}\left(\begin{array}{c}
			{p-n+1} \\
			r
		\end{array}\right) \dfrac{(3^{n+r-2}-1)(2^{n+r-2}-1)B_rB_{p-n-r+1}}{(n-1)(n+r-1)6^{n+r-2}} \pmod p & \text {if $2 \mid n$} , \\
		\left(1-\dfrac{(3^{n-1}-1)(2^{n-1}-1)}{2n 6^{n-1}}\right) B_{p-n} \pmod p & \text {if $2 \nmid n$}.\end{cases}
	\end{aligned}
\end{multline}
Combining (3.7) and (3.4), we comgplete the proof of Theorem 1.3.\qedsymbol
\normalsize
\baselineskip=17pt
\bibliographystyle{plain}
\bibliography{newcongruence}
\end{document}